\documentclass{amsart}
\usepackage{array}
\usepackage{amssymb,amsthm,amsmath}
\usepackage{epigraph}
\usepackage{ulem}
\usepackage[T1,T2A]{fontenc}
\usepackage[utf8]{inputenc}
\usepackage[russian,english]{babel}
\usepackage{mathtools}
\usepackage{etoolbox}

\usepackage[backend=bibtex,citestyle=numeric,bibstyle=numeric]{biblatex}
\makeatletter
\newcounter{qrr@oldeq}
\newcounter{qrr@oldsubeq}
\newcounter{qrr@realeq}
\renewenvironment{subequations}{%
  \refstepcounter{equation}%
  \protected@edef\theparentequation{\theequation}%
  \setcounter{parentequation}{\value{equation}}%
  \setcounter{equation}{0}%
  \def\theequation{\theparentequation\alph{equation}}%
  \ignorespaces
}{%
  \setcounter{qrr@oldeq}{\value{parentequation}}%
  \setcounter{qrr@oldsubeq}{\value{equation}}%
  \setcounter{equation}{\value{parentequation}}%
  \ignorespacesafterend
}
\newenvironment{subequations*}{%
  \setcounter{qrr@realeq}{\value{equation}}%
  \let\theparentequation\theequation%
  \patchcmd{\theparentequation}{equation}{parentequation}{}{}%
  \setcounter{parentequation}{\numexpr\value{qrr@oldeq}-1}%
  \setcounter{equation}{\value{qrr@oldsubeq}}%
  \def\theequation{\theparentequation\alph{equation}}%
  \refstepcounter{parentequation}%
  \ignorespaces
}{%
  \setcounter{qrr@oldeq}{\value{parentequation}}%
  \setcounter{qrr@oldsubeq}{\value{equation}}%
  \setcounter{equation}{\value{qrr@realeq}}%
  \ignorespacesafterend
}
\makeatother

\newtheorem{clm}{Claim}
\newtheorem{prop}[clm]{Proposition}
{ \theoremstyle{remark} \newtheorem{rem}[clm]{Remark}}

\newtheorem{defn}[clm]{Definition}
{\theoremstyle{definition} \newtheorem{ex}[clm]{Example}}

 \newcommand{\RR}{\mathbb{R}}
 
\newcommand{\ZZ}{\mathbb{Z}}

\DeclarePairedDelimiter{\bracepair}{\lbrace}{\rbrace}
\DeclarePairedDelimiter{\anglepair}{\langle}{\rangle}
\DeclarePairedDelimiter{\parenpair}{(}{)}
\DeclarePairedDelimiter{\vertpair}{\vert}{\vert}
\DeclarePairedDelimiter{\Vertpair}{\Vert}{\Vert}

\newcommand{\abs}[1]{\vertpair*{#1}}
\newcommand{\rnin}[2]{\anglepair*{#1, #2}}

\newcommand{\foursh}[1]{\widetilde{#1}}
\newcommand{\charfcn}[1]{\chi_{#1}}
\newcommand{\asimplex}[1]{\sigma \parenpair*{#1}}
\newcommand{\ansimplex}[2]{\sigma \parenpair*{#2;\, #1}}
\newcommand{\lowconst}{\mathcal{D}_1(n)}
\newcommand{\highconst}{\mathcal{D}_2(n)}
\newcommand{\card}[1]{\# \bracepair*{#1}}

\begin{document}
\normalem
\everymath{\displaystyle}

\begin{abstract}
It is shown that two conditions $f(a + \cdot) - f(\cdot)\in L^p(\RR)$, and $(\sin b\cdot) f(\cdot) \in L^p(\RR)$ guarantee $f \in L^p(\RR)$, $1 \leq p < \infty$, if and only if $ab$ is not in $(\pi \ZZ)$.
\end{abstract}

\title{Shifts of a Measurable Function and Criterion of $p$-integrability}
\author{Boris S. Mityagin}
\email{mityagin.1@osu.edu}
\address{231 West 18th Avenue, The Ohio State University, Columbus, OH 43210}
\maketitle

\footnotetext{Submitted to Journal of Approximation Theory}

\epigraph{\foreignlanguage{russian}{Когда б вы знали, из какого сора\\
      Растут стихи, не ведая стыда,\\
      Как желтый одуванчик у забора,\\
      Как лопухи и лебеда.}}{\foreignlanguage{russian}{А. Ахматова, `Тайны ремесла'}\footnotemark}
      \footnotetext{English translation is given in Section~\ref{sect:translation}}

\section{}  Let a measurable function $f$ on $\RR = (-\infty, \infty)$ have properties
\begin{subequations}
\begin{align}
\forall \, t \in \RR, && f(t + \cdot) - f(\cdot) \in L^2(\RR), \label{eq:shiftsin}\\
\intertext{and}
\forall \, s \in \RR, && \sin (s \cdot) f(\cdot) \in L^2(\RR). \label{eq:sinesin}
\end{align}
\end{subequations}
If a Fourier transform $\foursh{f}$ is reasonably defined then \eqref{eq:sinesin} is equivalent to \eqref{eq:shiftsin} for $\foursh{f}$.
\begin{clm} \label{clm:shifts} Under conditions \eqref{eq:shiftsin}, \eqref{eq:sinesin} we have $
f \in L^2(\RR)$.
\end{clm}

Recently, A. M. Vershik brought attention of the 25th St. Petersburg Summer Meeting in Mathematical Analysis, June 25 -- 30, 2016, to Claim~\ref{clm:shifts}.  He recalled that the known proof ``was done in terms of representation theory (of Heisenberg group) many years ago'' but noted that ``the simple proof still does not exist'' and after many years it is important ``to give a simple and direct proof.''\footnote{The presentation \cite{Vershik} gives a more extended motivation and links to the uncertainty principle although no reference to a published source is given.}
  A stronger form of Claim~\ref{clm:shifts} and its elementary proof was given just during the Meeting's session of A. Vershik's talk on June 30.  It is presented in Section~\ref{sect:proof}.  If the reader wants a proof only of Claim~\ref{clm:shifts}, there is no need to go beyond Section~\ref{sect:proof}.   

\section{} \label{sect:proof}
\begin{clm} \label{clm:specialshift} Let $f$ be a measurable function on $\RR$, and the following two conditions hold:
\begin{subequations}
\begin{gather}
\Delta(x) = f\left(x + \frac{\pi}{2} \right) - f(x) \in L^2(\RR). \label{eq:specificshift} \\
f(x) \cdot \sin(x) \in L^2(\RR) \label{eq:specificsine}
\end{gather}
\end{subequations}
Then $f \in L^2(\RR)$.
\end{clm}
\begin{proof}
Put $E = \bracepair{x: \abs{x - k\pi} \leq 10^{-6} \, \text{ for some } \, k \in \ZZ}$.
Then $
\abs{\frac{1}{\sin x}} \leq \frac{1}{\sin \delta} \leq 10^7 \text{ on } E^{\complement} = \RR \setminus E, \quad \delta = 10^{-6},$
so
\begin{equation} \label{eq:foncompisell2}
\left. f \right\vert {E^{\complement}} = (\sin x \cdot f(x) ) \cdot \frac{1}{\sin x} \in L^2(E^{\complement}).
\end{equation}
With $2\delta < \frac{\pi}{2}$ we have $E + \frac{\pi}{2} \subset E^{\complement}$ and $f(x) = f(x + \frac{\pi}{2}) - \Delta(x)$ for $x \in E$;
therefore $\Vertpair{f \vert E} \leq \Vertpair{f \vert E^{\complement}} + \Vertpair{\Delta} < \infty$,
and together with \eqref{eq:foncompisell2} and \eqref{eq:specificshift} we have $f \in L^2(\RR)$.
\end{proof}

\section{} Section~\ref{sect:proof} is an almost stenographic recording of what I have said at the Meeting's June 30 session.  Now we will talk about a more general setting (sorry, some repetition is unavoidable) and get negative results (Proposition~\ref{prop:genconclusion} and Example~\ref{example:lattices}) as well.  Of course, $L^2$--norm is not special in our analysis in Section~\ref{sect:proof}.  Instead of $L^2$ we can talk about any Banach space $X$ of measurable functions on $\RR$ with two properties:
\begin{subequations}
\begin{align}
g \in X & \Rightarrow g(\cdot + t) \in X, \quad t \in \RR \label{eq:genershifts}\\
g \in X & \Rightarrow g \cdot h \in X, \quad \forall h \in L^{\infty} (\RR). \label{eq:generbdds}
\end{align}
\end{subequations}
Moreover, we do not need global conditions \eqref{eq:shiftsin}, \eqref{eq:sinesin}; just a pair $(t; s) = \left( \frac{\pi}{2}; 1 \right)$ with \eqref{eq:specificshift}, \eqref{eq:specificsine} holding was good enough for the proof in Section~\ref{sect:proof}.  More general than Claim~\ref{clm:specialshift} is true:
\begin{prop} \label{prop:generelltwo}
Let $X$ be a Banach space of measurable functons on $\RR$ with properties \eqref{eq:genershifts}, \eqref{eq:generbdds} and
\begin{subequations*}
\begin{align} \label{eq:shiftall}
g \in X \Rightarrow g(a \cdot) \in X, \quad  \forall a \neq 0.
\end{align}
\end{subequations*}
Let $(t, s)$ be two real non-zero numbers such that
\begin{equation} \label{eq:nonsingular}
st \neq k \pi, k \in \ZZ.
\end{equation}
If a measurable function $f$ on $\RR$ satisfies conditions
\begin{subequations}
\begin{align}
\Delta_t(x) = f(t + x) - f(x) \in X, \quad \text{ and } \label{eq:deltatshift}\\
(\sin sx) \cdot f(x) \in X, \label{eq:sinesshift}
\end{align}
\end{subequations}
then 
\begin{equation} \label{eq:banachconclusion}
f \in X.
\end{equation}
\end{prop}

\section{}

\begin{proof}[Proof of Proposition~\ref{prop:generelltwo}]
The assumption \eqref{eq:shiftall} permits us to rescale a variable $x$ and go to $F(\xi) = f\parenpair*{\frac{\xi}{s}}$.  It brings us to the pair $(T; 1) = (ts; 1)$ instead of $(t, s)$, so we need to prove that $F \in X$ under the assumptions
\begin{subequations}
\begin{align}
F(\xi + T) - F(\xi) \in X \label{eq:bigFshifts}
\intertext{and}
\sin \xi \cdot F(\xi) \in X, \label{eq:bigFsine}
\end{align}
\end{subequations}
with
\begin{equation} \label{eq:tnotpi}
T \neq k \pi, \quad k \in \ZZ.
\end{equation}
We can choose $\tau$ and $m$ such that 
\begin{equation} \label{eq:tshiftdef}
T = m \pi + \tau, \quad m \in \ZZ, \quad 0 < \tau < \pi,
\end{equation}
and $\delta > 0$,
\begin{equation} 
2 \delta \leq \tau \leq \pi - 2\delta. \label{eq:deltasize}
\end{equation}
Put $E = \bigcup_{k \in \ZZ} [k\pi - \delta; k \pi + \delta]$ and
\begin{equation} \label{eq:hdef}
h(\xi) = \begin{cases} \frac{1}{\sin(\xi)}, & \xi \in E^{\complement} = \RR \setminus E,\\
0, & \xi \in E \end{cases}
\end{equation}
so $h \in L^{\infty}$ with $\Vertpair{h \vert L^{\infty}} \leq \frac{1}{\sin \delta}$.  Then
\begin{equation} \label{eq:FonEcomp}
F(\xi) \cdot \charfcn{E^{\complement}} = \left[ \sin \xi \cdot F(\xi) \right] \cdot h(\xi)
\end{equation}
is in $X$ by \eqref{eq:bigFsine} and \eqref{eq:generbdds}.  \eqref{eq:deltasize} guarantees that $E + T = E + \tau \subset E^{\complement}$.  
Now we use \eqref{eq:bigFshifts} and the identity $F(\xi) = F(\xi + T) - \Delta(\xi)$, with $\Delta(\xi) = F(\xi + T) - F(\xi)$ in $X$ by \eqref{eq:bigFshifts} to conclude
that both $F(\xi + T) \sin \xi = F(\xi) \sin \xi + \Delta (\xi) \sin \xi$ and its shift
\begin{equation} \label{eq:critshift}
F(\xi) \sin(\xi - T)
\end{equation}
are in $X$.  
By \eqref{eq:tshiftdef}, \eqref{eq:deltasize}  $\min_{\xi \in E} \abs{\sin(\xi - T)} = \min_{\xi \in E} \abs{\sin(\xi - \tau)} \geq \sin \delta.$  Put --- compare \eqref{eq:hdef} ---
\begin{equation*} \label{eq:HHdef}
H(\xi) = \begin{cases} \frac{1}{\sin(\xi - T)}, &\quad \xi \in E\\
0, &\quad \xi \in E^{\complement} \end{cases}
\end{equation*}
so $H \in L^{\infty}(\RR)$, $\Vertpair{H \vert L^{\infty}} \leq \frac{1}{\sin \delta}$.    Therefore,
\begin{equation} \label{eq:FonE}
F(\xi) \cdot \charfcn{E}(\xi) = [\sin(\xi - T) F(\xi) ] \cdot H(\xi)
\end{equation}
and by \eqref{eq:critshift} and \eqref{eq:bigFsine} the function \eqref{eq:FonE} is in $X$.  Together with \eqref{eq:FonEcomp} this observation completes the proof.
\end{proof}

\section{}  This is worth to notice that the quantization condition \eqref{eq:nonsingular} or \eqref{eq:tnotpi} is crucial.  Indeed, for a pair $(T, 1)$, $T = \pi$, now we'll construct a function $f(x)$ such that
\begin{subequations}
\begin{align}
\int \abs{f(x)}^2 \, dx & = \infty. \label{eq:notelltwo}\\
f(x) \cdot \sin x = g(x) &\in L^2 \label{eq:onesine}\\
f(x) - f(x + \pi) & \in L^2 \label{eq:oneshift}
\end{align}
\end{subequations} 
It will be bad for any $T = m \pi$, $m \in \ZZ$, of course.
Put 
\begin{equation} \label{eq:akdef}
a_0 = \frac{1}{4}, \qquad a_k =  \frac{1}{5\abs{k}}, \quad k \neq 0
\end{equation}
and 
\begin{equation} \label{eq:fdef}
f(x) = \begin{cases} 1, & x \in I_k, \quad k \in \ZZ\\
0, & \text{otherwise.} \end{cases}
\end{equation}
Then by \eqref{eq:akdef}, \eqref{eq:fdef}
\[
\int \abs{f(x)}^2 \, dx = \sum_{k \in \ZZ} \int_0^{a_k} 1 \, dt = \infty,
\]
but with
\begin{equation} \label{eq:ineqs}
\sin x \leq x, \quad 0 \leq x \leq \frac{1}{4}, \, \text{ and } \abs{\sin(x + j\pi)} = \abs{\sin x} \, \, \forall j \in \ZZ,
\end{equation}
\[
\begin{split}
\int \abs{g(x)}^2 \, dx &= \sum_{k \in \ZZ} \int_0^{a_k} (\sin t)^2 \, dt \\
& \leq \sum_{k \in \ZZ} \int_0^{a_k} t^2 \, dt = \frac{1}{3} \sum_{k \in \ZZ} a_k^3 < \infty.
\end{split}  
\]
We still need to check \eqref{eq:oneshift}.  First, we analyze the case $x \geq 0$, $k \geq 0$.  

On $I_k \cap (I_{k + 1} - \pi) = I_{k + 1} - \pi$
\begin{equation} \label{eq:pishiftcancel}
\begin{split}
f(x) - f(x + \pi) = 0
\end{split}
\end{equation}
but on $I_k \setminus( I_{k + 1} - \pi) = k\pi + [a_{k + 1}, a_k]$
\[
f(x)- f(x + \pi) = f(x) = 1,
\]
and equals $0$ otherwise.  If $x < 0, k < 0$, $(I_{k + 1} - \pi) \supset I_k$ so 
\[
f(x) - f(x + \pi) = \begin{cases} 0 & \text{ on } I_k,\\
-1 & \text{ on } (I_{k + 1} - \pi) \setminus I_k = k\pi + [a_k, a_{k + 1}],\\
0 & \text{otherwise.} \end{cases}
\]

Therefore, by \eqref{eq:akdef}
\[
\begin{split}
\int \abs{f(x) - f(x + \pi)}^2 \, dx & \leq 2 \sum_{k \geq 0} \left[ \int_{a_{k + 1}}^{a_k} 1 \, dt \right]\\
& = 2 a_0 = \frac{1}{2}.
\end{split}
\]

\begin{rem} \label{rem:ellptoo} The same example, i.e., a function \eqref{eq:fdef}, is good to show that \eqref{eq:onesine}, \eqref{eq:oneshift}, \eqref{eq:notelltwo} holds if $L^2$ is changed to $L^p(\RR)$, $1 \leq p < \infty$.
\end{rem}
The function \eqref{eq:fdef} gives a counterexample for a pair $(\pi, 1)$, so it is good for any $(t, s) \in (\pi \ZZ) \times \ZZ$ if $t \neq 0$, $s \neq 0$.  Together with Proposition~\ref{prop:generelltwo} this observation implies the following.
\begin{prop} \label{prop:genconclusion}
Let $X = L^p(\RR)$, $1 \leq p < \infty$.  Then conditions \eqref{eq:deltatshift}, \eqref{eq:sinesshift} imply \eqref{eq:banachconclusion} if and only if $ts \not\in \pi \ZZ$.  
\end{prop}
\begin{proof}
To complete the proof we need to give counterexamples when $ts = 0$.  If $t = 0$ the condition \eqref{eq:deltatshift} trivially holds for any measurable $f$, and \eqref{eq:sinesshift} is satisfied if we put
\[
f(x) = \begin{cases} \frac{1}{x}, &\quad 0 < x < 1 \vspace{0.5 ex} \\
0, & \text{otherwise.} \end{cases}
\]
But $f$ is not in $X = L^p$.  If $s = 0$ put $g(x) = (1 + \abs{x})^{-a}$, $\frac{1}{p} - 1 < a < \frac{1}{p}$, or just $g(x) = 1$; then \eqref{eq:deltatshift}, \eqref{eq:sinesshift} hold for $g$ but $g$ is not in $X = L^p$.
\end{proof}

\section{}  With some adjustments, we can give the multidimensional analogs of Propositions~\ref{prop:generelltwo} and  \ref{prop:genconclusion}.

Let $Y$ be a Banach space of measurable functions on $\RR^n$ with properties
\begin{subequations}
\begin{align}
g \in Y & \Rightarrow g(x + t) \in Y \quad \forall t \in \RR^n \label{eq:multidshift}\\
g \in Y & \Rightarrow g \cdot h \in Y, \quad \forall h \in L^{\infty}(\RR^n) \label{eq:multidmult}\\
g \in Y & \Rightarrow g(Cx) \in Y \quad \text{for any $(n \times n)$--matrix $C$, } \det C \neq 0. \label{eq:multiddilate}
\end{align}
\end{subequations}
\begin{defn}
We say that a pair $A, B \subset \RR^n$ of subsets is \textit{assertive in $Y$} if two conditions on a measurable function $F$ on $\RR^n$
\begin{subequations}
\begin{align}
F(\cdot + t) - F(\cdot) \in Y, \quad \forall t \in A, \label{eq:ashifts}\\
F(x) \cdot (\sin \rnin{x}{b}) \in Y, \quad \forall b \in B \label{eq:bsines}\\
\intertext{imply}
F(x) \in Y.  \label{eq:resulting}
\end{align}
\end{subequations} 
\end{defn} 
\begin{ex} \label{example:singletons} Take two singletons $A = \bracepair{a}$, $B = \bracepair{b}$, $a, b \in \RR^n$.  \textit{They are assertive if a scalar product $t \equiv \rnin{a}{b} \neq \pi m$, $m \in \ZZ$.}  Indeed, with $b \neq 0$ we can rescale by \eqref{eq:multiddilate} a variable $x$ and assume
\begin{equation} \label{eq:multidStandardized}
b = e_1, a = te_1 + a^{\prime}, \quad \rnin{a^{\prime}}{e_1} = 0.  
\end{equation}
With $\delta = \frac{1}{4} \tau$, $\tau = \min \bracepair*{ \abs{t - k \pi}: k \in \ZZ}$ define
\[
E = \bracepair*{x \in \RR^n: \abs{(\sin x_1)} < \delta}
\]
and notice that $E + a \subseteq E^{\complement}$.  We can repeat (with proper adjustment) the proof of Proposition~\ref{prop:generelltwo}.

If, however, $t \equiv \rnin{a}{b} = m\pi$, $m \in \ZZ$, $A$, $B$ are not assertive in $L^p(\RR^n)$, $1 \leq p < \infty$  -- see Claim~\ref{clm:singletons} below.
\end{ex} 
\section{} 
\begin{ex} \label{example:lattices}
Now take
\[
A = \alpha \ZZ^n, B = \beta \ZZ^n.
\]
\textit{This pair is assertive in $L^p(\RR^n)$, $1 \leq p < \infty$, if and only if $t = \alpha \cdot \beta \not\in \pi \ZZ$.}

Just two points $\alpha e_1 \in A$, $\beta e_1 \in B$ guarantee (after Example~\ref{example:singletons}) that $A$, $B$ are assertive if $t \neq \pi m$, $m \in \ZZ$.  

Now we will construct a bad function $F$ for $\alpha = \pi$, $\beta = 1$, i.e., such $F$ on $\RR^n$ that
\begin{subequations} \label{eq:multidBad}
\begin{align}
 \int\limits_{\RR^n} \abs{F(x)}^p \, dx &= \infty, \label{eq:multidFbad}\\
F(x) \cdot (\sin x_j) &\in L^p(\RR^n), \, \, 1 \leq j \leq n \label{eq:multidFsingood}\\
\int\limits_{\RR^n} \abs{F(x) - F(x + \pi e_j)}^p \, dx&  < \infty, \, \, 1 \leq j \leq n \label{eq:multidShiftgood}
\end{align}
\end{subequations}
It will be bad for $A$ and $B$ if we will make two observations.  
\begin{equation} \label{eq:multidFsineStrong}
\text{Conditions \eqref{eq:multidFsingood} guarantee that } F(x) \cdot (\sin \rnin{b}{x}) \in L^p(\RR^n) \text{ for any } b \in \ZZ^n \tag{\ref{eq:multidFsingood}$^{\prime}$}  
\end{equation}
Indeed, by induction, one can explain that 
\[
\sin \rnin{b}{x} = \sum_{j = 1}^n Q_j^b(x) \sin(x_j),
\]
where $Q_j^b$ are trigonometric polynomials on $\RR^n$ of period $2\pi$, i.e., algebraic polynomials of $\cos x_j$, $\sin x_j$, $1 \leq j \leq n$.  They are bounded, i.e., $Q_j^b(x) \in L^{\infty}(\RR^n)$, so with 
\[
F(x) \cdot \left( \sin \rnin{b}{x} \right) = \sum_{j = 1}^n Q_j^b(x) [F(x) \cdot \sin(x_j)],
\]
\eqref{eq:multidFsingood} implies  \eqref{eq:multidFsineStrong}.  
\end{ex}
Any $a \in A$ is a finite sum of vectors $\pm \pi e_j$, $1 \leq j \leq n$, so
\begin{equation} \label{eq:multidFshiftStrong}
\text{the conditions \eqref{eq:multidShiftgood} guarantee that } F(x) - F(x + a) \in L^p(\RR^n) \, \, \forall a \in A = \pi \ZZ^n. \tag{\ref{eq:multidShiftgood}$^{\prime}$} 
\end{equation}
Now we focus on \eqref{eq:multidFbad} -- \eqref{eq:multidFshiftStrong}.  

We now present some preliminary facts about the construction blocks.  Define
\begin{equation} \label{eq:simplexdef}
\ansimplex{a}{n} = \asimplex{a} = \bracepair*{\xi \in \RR^n: 0 \leq \xi_j, \, \, 1 \leq j \leq n; \quad \sum_{j = 1}^n\xi_j \leq a}, \quad a > 0.
\end{equation}
Then
\begin{equation} \label{eq:simplexCharInt}
\int\limits_{\asimplex{a}} 1 d^n \xi = \frac{1}{n!} a^n ,
\end{equation}
and
\begin{equation} \label{eq:simplexVarInt}
\begin{split}
\int\limits_{\asimplex{a}} \xi_j^p d^n \xi &= \int_0^a t^p \, dt \int\limits_{\ansimplex{a-t}{n-1}} d^{n-1} t^{\prime} \\
& \leq \frac{1}{(n-1)!} \int_0^a t(a-t)^{n-1} \, dt \\
& = \frac{1}{(n-1)!} a^{n + 1} \int_0^1 t(1 - t)^{n - 1} \, dt = \frac{1}{(n + 1)!} a^{n + 1}.
\end{split}
\end{equation}
With notations 
\[
\kappa = (k_1, \dotsc , k_n) \in \ZZ^n, k(\kappa) \equiv k = \sum_{j = 1}^n \abs{k_j},
\]
let us observe the following:
\begin{equation}\label{eq:simplexsurfaceasymp}
\lowconst (k + 1)^{n - 1} \leq \card{\kappa \in \ZZ^n_+ : \sum_{j = 1}^n k_j = k} \leq \highconst (k + 1)^{n-1} \, \, \text{ for all } k \geq 1,
\end{equation}
where constants $0 < \lowconst < \highconst$ do not depend on $k$.
If $n = 2$ the number $\# = k + 1$.  Thus we can by induction explain \eqref{eq:simplexsurfaceasymp} for any $n$.  

Let $h(\kappa)$ be a positive function on $\ZZ^n$ such that
\begin{equation} \label{eq:hNormalize}
h(\kappa) = H(1 + k(\kappa)), \, \, \text{ where $H$ is a function on }[1, \infty).
\end{equation}  
Then
\begin{equation} \label{eq:hLowbd}
\sum_{\kappa \in \ZZ^n} h(\kappa) \geq \sum_{\kappa \in \ZZ^n_+} h(\kappa) \geq \lowconst \sum_{k = 0}^{\infty} (1 + k)^{n-1} H(1 + k),
\end{equation}
and
\begin{equation} \label{eq:hHighbd}
\sum_{\kappa \in \ZZ^n} h(\kappa) \leq 2^n \sum_{\kappa \in \ZZ_+^n} h(\kappa) \leq 2^n \highconst \sum_{k = 0}^{\infty} (1 + k)^{n-1} H(1 + k).  
\end{equation}
Now choose
\begin{equation} \label{eq:rDef}
r(\kappa) = R(1 + k), \quad R(x) = \left( \frac{1}{x} \right)^{\gamma}, \quad 1 \geq \gamma > 1 - \frac{1}{n + 1}.
\end{equation}
and define
\begin{equation} \label{eq:multidFdef}
F(x) = \begin{cases} 1, & x \in \bigcup_{\kappa \in \ZZ^n} I(\kappa), \quad I(\kappa) = \bracepair{\pi \kappa + \asimplex{r(\kappa)}},\\
0, & \text{ otherwise.} \end{cases} 
\end{equation}
We claim that \eqref{eq:multidFbad} -- \eqref{eq:multidFshiftStrong} hold.  Indeed, by \eqref{eq:simplexCharInt} and \eqref{eq:hLowbd}, \eqref{eq:rDef}
\begin{equation} \label{eq:multidFestimate}
\int\limits_{\RR^n} F(x)^p \, dx = \sum_{\kappa \in \ZZ^n} \int\limits_{\asimplex{r(\kappa)}} 1 \, dx = \sum_{\kappa \in \ZZ^n} r^n(\kappa) \cdot \frac{1}{n!} \geq \frac{\lowconst}{n!} \sum_{k = 0}^{\infty} (1 + k)^{n - 1} \left( \frac{1}{1 + k)} \right)^{\gamma n} = \infty.
\end{equation}
To check \eqref{eq:multidFsingood} let us notice that
\[
\frac{\sin t}{t} \leq 1, \quad 0 \leq t \leq \pi, \text{ and } \abs{\sin(t + \pi \ell)} = \abs{\sin t}, \, \, \forall \ell \in \ZZ.
\]
Therefore, by \eqref{eq:simplexVarInt} and \eqref{eq:hHighbd}
\begin{equation} \label{eq:multidSineestimate}
\begin{split}
\int\limits_{\RR^n} \abs{(\sin x_j) F(x)}^p \, dx & \leq \sum_{\kappa \in \ZZ^n} \int\limits_{\asimplex{r(\kappa)}} \abs{\sin(x_1)}^p \, d^n x \\
& = \frac{1}{(n + 1)!} \sum_{\kappa \in \ZZ^n} r(\kappa)^{n + 1} \\
& \leq \frac{2^n \highconst}{(n + 1)!} \sum_{k = 0}^{\infty} (1 + k)^{n-1} \cdot \left( \frac{1}{k + 1} \right)^{\gamma(n + 1)} < \infty 
\end{split} 
\end{equation}
if $n - 1 - \gamma(n + 1) < -1$, i.e., $\gamma > 1 - \frac{1}{n + 1}$.  This is the condition \eqref{eq:rDef} so \eqref{eq:multidFsingood} holds.  

Finally, by \eqref{eq:simplexCharInt} and \eqref{eq:multidFdef}, \eqref{eq:hHighbd}
\begin{equation} \label{eq:multidShiftestimate}
\begin{split}
\int\limits_{\RR^n} \abs{F(x) - F(x + \pi e_j)}^p \, dx & = \sum_{\kappa \in \ZZ^n} \frac{1}{n!} \abs{r(\kappa)^n - r^n(\kappa + \pi e_j)}\\
& \leq \frac{2^n}{n!} \sum_{\kappa \in \ZZ^n_+} \left( r(\kappa) - r(\kappa + \pi e_1)^n \right)\\
& = \frac{2^n}{n!} \sum_{\kappa \in \ZZ^n_+} \left[ R^n(1 + k(\kappa)) - R^n(k(\kappa) + 2) \right]\\
& = \frac{2^n}{n!} \sum_{\kappa \in \ZZ^n_+} \left[ \left( \frac{1}{1 + k(\kappa)} \right)^{\gamma n} - \left( \frac{1}{2 + k(\kappa)} \right)^{\gamma n} \right]\\
& \leq \frac{2^n \gamma}{(n - 1)!} \sum_{\kappa \in \ZZ^n_+} \left( \frac{1}{1 + k(\kappa)} \right)^{\gamma n + 1}\\
& \leq \frac{2^n \gamma \highconst}{(n-1)!} \sum_{k = 0}^{\infty} (1 + k)^{n - 1} \left( \frac{1}{1 + k} \right)^{\gamma n + 1} < \infty  
\end{split}
\end{equation}
if $n - 1 - (\gamma n + 1) = n(1 - \gamma) - 2 < -1,$ i.e., $\gamma > 1 - \frac{1}{n}$.  This condition follows from \eqref{eq:rDef} so \eqref{eq:multidShiftgood} holds.  

\section{}  After analysis of Example~\ref{example:lattices} we can explain the following.  
\begin{clm} \label{clm:singletons} With notation of Example~\ref{example:singletons}, if $t = \rnin{a}{b} = m \pi$, $m \in \ZZ$, there exists a measurable function $F(x)$ on $\RR^n$ such that
\begin{subequations}
\begin{align} 
\int\limits_{\RR^n} \abs{F(x)}^p \, dx &= \infty, \label{eq:singlesFbad} \\
F(x) \cdot \left( \sin \rnin{x}{b} \right) &\in L^p(\RR^n), \label{eq:singlesFsinegood} \\
\int\limits_{\RR^n} \abs{F(x) - F(x + a)}^p \, dx &< \infty. \label{eq:singlesFshiftsgood}
\end{align}
\end{subequations}
\end{clm}
\begin{proof}
If vectors $a, b \in \RR^n$ are linearly dependent we can assume (compare a rescaling in Proposition~\ref{prop:generelltwo}; use \eqref{eq:multiddilate} if necessary) that $a = m\pi e_1$, $b = e_1$.  Then a function
\[
F(x) = f(x_1) \cdot \varphi(x^{\prime}), \quad x^{\prime} = (x_2, \dotsc, x_n),
\]
where $f \in \eqref{eq:fdef}$ and $\varphi \in L^p(\RR^{n-1})$, say, $\varphi = \left( 1 + \sum_{j = 2}^n x_j^2 \right)^{-n}$, satisfies the conditions \eqref{eq:singlesFbad} -- \eqref{eq:singlesFshiftsgood}.
If $a$, $b$ are linearly independent, we can assume
\[
b = e_1, \quad a = m \pi e_1 + \tau e_2, \quad \tau > 0,
\]
and choose 
\[
F(x) = f(x_1, x_2) \varphi( x^{\prime \prime}),\quad  x^{\prime \prime} = (x_3, \dotsc, x_n),
\]
where $\varphi = \left( 1 + \sum_{j = 3}^n x_j^2 \right)^{-n}$ and 
\[
f = \begin{cases} 1, & (x_1, x_2) \in I_{\lambda}, \quad \lambda = (\ell_1, \ell_2) \in \ZZ^2\\
0, & \text{otherwise,} \end{cases}
\]
with
\[
\begin{gathered}
I_{\lambda} = \bracepair{\ell_1(\pi e_1) + \ell_2 (\tau e_2) + \ansimplex{r_{\lambda}}{2}},\\
r_{\lambda} = R(1 + \abs{\ell_1} + \abs{\ell_2}), \quad R(x) = \frac{1}{x}, \, \,  x \geq 1.
\end{gathered}
\]
All technicalities to explain \eqref{eq:singlesFbad} -- \eqref{eq:singlesFshiftsgood} are already done in Example~\ref{example:lattices}.  
\end{proof}
After Examples~\ref{example:singletons}, \ref{example:lattices} and Claim~\ref{clm:singletons} it would be interesting to describe all assertive (for $L^p$, $1 \leq p < \infty$) pairs $A, B \subset \RR^n$.  It seems reasonable to conjecture that \textit{the pair $A, B \subset \RR^n$ is assertive if and only if there are vectors $a \in A$, $b \in B$ such that $\rnin{a}{b} \not\in \pi \ZZ$.}  If $n = 1$ this is the statement of Proposition~\ref{prop:genconclusion}.  For any $n > 1$ the ``if'' is explained in Example~\ref{example:singletons}.

\section{Acknowledgements}  Part of this work was done when I was participating in the conference ``25th St. Petersburg Summer Meeting in Mathematical Analysis: Tribute to Victor Havin, 1933-2015'' ;  I thank the organizers for their support.  

I want also to thank Nikolai Nikolski who on June 30 chaired the Meeting's session and gave me the opportunity to present the above solution (Section~\ref{sect:proof}) during the session immediately after Vershik's talk.

Moreover, I am grateful to Maria Roginskaya and Nadia Zalesskaya, who helped me put in order the proof and its presentation while the speaker was talking, and to Haakan Hedenmalm, who \foreignlanguage{russian}{кричал из зала ``Давай подробности!''}, i.e. he insisted on my giving all the details.

\pagebreak

\section{Translation of the epigraph (as given in \cite{Akhmatova})} \label{sect:translation} \mbox{} \\
If you only knew what kind of trash\\
Poems shamelessly grow in:\\
Like weeds under the fence,\\ 
Like crabgrass, dandelions.\\
A. Akhmatova, ``Secrets of the Trade.''
\nocite{*}

\printbibliography

\end{document}